\documentclass[12pt, a4paper, leqno]{amsart}
\usepackage[utf8]{inputenc}
\usepackage{amsmath}
\usepackage{amsfonts}
\usepackage{amssymb}
\usepackage{times}
\usepackage[english]{babel}
\usepackage[T1]{fontenc} 
\usepackage{url} 
\usepackage{color,esint}
\usepackage[dvipsnames]{xcolor}
\usepackage[colorlinks, allcolors=RedViolet,pdfstartview=,pdfpagemode=UseNone]{hyperref} 
\usepackage{graphicx} 
\usepackage{enumitem}
\usepackage{tabularx}
\usepackage{mathtools}
\usepackage{pgf,tikz}
\usepackage{mathrsfs}
\usetikzlibrary{arrows}

\let\oldmarginpar\marginpar
\renewcommand\marginpar[1]{\-\oldmarginpar[\raggedleft\footnotesize #1]%
{\raggedright\footnotesize #1}}

\usepackage{amsthm}
\theoremstyle{plain}
\newtheorem{thm}[equation]{Theorem}
\newtheorem{lem}[equation]{Lemma}
\newtheorem{prop}[equation]{Proposition}

\theoremstyle{definition}
\newtheorem{defn}[equation]{Definition}
\newtheorem{assumptions}[equation]{Assumptions}

\theoremstyle{remark}
\newtheorem{rem}[equation]{Remark}

\numberwithin{equation}{section}

\newcommand{\R}{\mathbb{R}}

\newcommand{\Rn}{{\mathbb{R}^n}}

\renewcommand{\phi}{\varphi}
\def\le{\leqslant}
\def\leq{\leqslant}
\def\ge{\geqslant}
\def\geq{\geqslant}
\def\phi{\varphi}
\def\rho{\varrho}
\def\vartheta{\theta}

\newcommand{\Phiw}{\Phi_{\text{\rm w}}}
\newcommand{\Phic}{\Phi_{\text{\rm c}}}
\newcommand{\Phis}{\Phi_{\text{\rm s}}}

\def\loc{{\rm loc}}
\def\BMO{\operatorname{BMO}}

\newcommand{\inc}[1]{\hyperref[def:aInc]{{\normalfont(Inc){\ensuremath{_{#1}}}}}}
\newcommand{\dec}[1]{\hyperref[def:aDec]{{\normalfont(Dec){\ensuremath{_{#1}}}}}}
\newcommand{\ainc}[1]{\hyperref[def:aInc]{{\normalfont(aInc){\ensuremath{_{#1}}}}}}
\newcommand{\ainci}[1]{\hyperref[def:aInc]{{\normalfont(aInc){\ensuremath{_{#1}^\infty}}}}}
\newcommand{\adec}[1]{\hyperref[def:aDec]{{\normalfont(aDec){\ensuremath{_{#1}}}}}}
\newcommand{\adeci}[1]{\hyperref[def:aDec]{{\normalfont(aDec){\ensuremath{_{#1}^\infty}}}}}
\newcommand{\azero}{\hyperref[def:a0]{{\normalfont(A0)}}}
\newcommand{\aone}{\hyperref[def:a1]{{\normalfont(A1)}}}

\newcommand{\aones}[1]{\hyperref[def:a1s]{{\normalfont(A1-{\ensuremath{{#1}})}}}}

\newcommand{\atwo}{\hyperref[def:a2]{{\normalfont(A2)}}}

\date{\today}

\begin{document}

\title
[Fractional operators on generalized Orlicz spaces]
{Fractional operators and their commutators on generalized Orlicz spaces}
\author[A. Karppinen]{Arttu Karppinen}


\begin{abstract}
In this paper we examine boundedness of fractional maximal operator. The main focus is on commutators and maximal commutators on generalized Orlicz spaces for fractional maximal functions and Riesz potentials. We prove their boundedness between generalized Orlicz spaces and give a characterization for functions of bounded mean oscillation. To best of our knowledge, these results are also new in the special case of double phase spaces.
\end{abstract}

\maketitle


\section{Introduction}

The Hardy--Littlewood maximal operator is one of the most central operators in modern harmonic analysis and theory of partial differential equations. Its boundedness properties have been deeply connected to various Sobolev--Poincar\'e type inequalities and density of smooth functions in spaces related to non-standard growth conditions, for instance. For that reason, a systematic study of various variational integrals or partial differential equations under non-standard growth conditions has often started with proving boundedness of the maximal operator in a related function space, see for instance \cite{ColM15a, Has15}.

The maximal function has also a fractional variant introduced in \cite{fractional1} and its boundedness properties have also been intensively studied. This version is also related to fractional integrals such as Riesz potentials. Chianhi \cite{Cia99} first proved necessary and sufficient conditions for boundedness of fractional maximal operator between different Orlicz spaces. Here the condition between Orlicz functions was described by ''global domination''. In \cite{orlicz1} this condition was proven to be equivalent to presentation used also in this article. The boundedness results have also been investigated in the variable exponent case, see \cite{fractional3,CruDF08, fractional2}.

On the other hand, a boundedness of commutators $[T,b]f = T(bf) - bTf$ of various operators $T$ and functions $b$ in harmonic analysis has been a widely studied topic. One of the most influential paper regarding commutators in harmonic analysis is by Coifman, Rochberg and Weiss \cite{CoiRW76}. They gave a characterization of $\BMO(\Rn)$ in terms of boundedness of commutator $[T,b]$ in $L^{p}(\Rn)$ when $T$ is a Calder\'on--Zygmund singular integral operator. Similar characterization of Lipschitz spaces was given by Janson \cite{Jan78}: $[T,b]$ is bounded from $L^{p}(\Rn)$ to $L^{q}(\Rn)$, where $1<p<q<\infty$ if and only if $b \in \Lambda_\beta$ and $\beta = n\left (\frac{1}{p} - \frac{1}{q}\right )$. These type of results have successfully been used for research regarding products of functions from real Hardy space $H^1(\Rn)$ and $\BMO(\Rn)$ \cite{product} and giving new characterization of well known function spaces such as $\BMO(\Rn)$, Lipschitz spaces and even Campanato spaces \cite{campanato}.

When the operator is a maximal function instead of a Calder\'on--Zygmund operator, the boundedness results of commutators were studied in \cite{BasMR01} and for fractional maximal function in \cite{Lp-fractional}. 
These results have had their generalizations from standard $L^{p}$-spaces to their non-standard counterparts such as variable exponent Lebesgue and Orlicz spaces \cite{orlicz1, varexp1, varexp2} and more refined Morrey type spaces \cite{Alm20,morrey2,morrey1}.

In this paper we study boundedness of fractional operators and their commutator variants in generalized Orlicz spaces (also known as Musielak--Orlicz spaces). We prove that similar structural assumptions required for boundedness of Hardy--Littlewood maximal operator in generalized Orlicz spaces also imply boundedness of the fractional variant. This is closely related to boundedness of Riesz potentials, which have already been studied in \cite{HarH_Riesz}. Here we give a slightly different framework for defining the target space and record that these results do not require the generating $\phi$-functions to be $N$-functions. Assumptions for these boundedness results are in turn inherited by the commutators and we use the boundedness results to give a characterization for functions of bounded mean oscillation $\BMO(\Rn)$. Also, commutators related to Riesz potentials are investigated.

Generalized Orlicz spaces include many of the widely studied function spaces such as classical $L^{p}$ spaces, Orlicz spaces, variable exponent spaces and double phase spaces as special cases. Developing theory in this general setting allows us to capture many delicate phenomena not present in $L^{p}$-scale and describe them with improved accuracy using assumptions intrinsic to the problem at hand. Main results achieved here are, to best of our knowledge, new already in the double phase situation. The methodology in this paper is to give transparent proofs without resorting to extrapolation arguments or other abstract arguments.

However, the improved generality comes with some penalties. The connection between boundedness of the maximal operator in a generalized Orlicz space and its dual space is not yet fully understood and therefore not utilized here. For this reason, the assumptions remain technical instead of just the boundedness assumption. However, Diening has obtainted some partial results in this direction \cite{Die05}.  Additionally, the assumption linking the domain and target space is known to be sharp in Orlicz spaces, but the non-autonomous nature of generalized Orlicz spaces leaves this as an open question. To best of our knowledge, this is the case already in variable exponent spaces.

The rest of the article is organized as follows. In the second section we define all related function spaces together with operators we study. Third section is devoted to fractional maximal operator. Here we prove its boundedness and connect our assumptions to previous study related to Riesz potential. In the fourth section we prove boundedness of commutators related to Riesz potential. The fifth section briefly discusses commutators of Hardy--Littlewood maximal function and proofs owe to pointwise estimates proved in \cite{AgcGKM15}. In the final section give boundedness results of commutators related to fractional maximal function and give a new characterization of $\BMO(\Rn)$.

\section{Preliminaries}

By $L$-almost increasing we mean that a function satisfies the inequality $f(s) \le L f(t)$ for all $s<t$ and some constant $L\ge 1$ and  $L$-almost decreasing is defined analogously.
If there exists a constant $C$ such that $f(x) \leq C g(x)$ for almost every $x$, then we write $f \lesssim g$. Additionally, we write $f \approx g$ if $f\lesssim g\lesssim f$ holds. In this article $\Omega \subset \Rn$ is an open set. In the case of a measurable set$A$, we denote its characteristic function by $\chi_A$ and its Lebesgue measure by $|A|$. We reserve $Q$ to mean any cube in $\Rn$ with sides parallel to coordinate axis and specify it to have a center $x$ and side-length $2r$, denoted as $Q(x,r)$, when needed.
\subsection{Generalized $\Phi$-functions}

\begin{defn}
We say that $\phi: \Omega\times [0, \infty) \to [0, \infty]$ is a 
\textit{weak $\Phi$-function}, and write $\phi \in \Phiw(\Omega)$, if 
the following conditions hold:
\begin{itemize}
\item 
For every measurable function $f:\Omega\to \R$ the function $x \mapsto \phi(x, f(x))$ is measurable 
and for every $x \in \Omega$ the function $t \mapsto \phi(x, t)$ is non-decreasing.
\item 
$\displaystyle \phi(x, 0) = \lim_{t \to 0^+} \phi(x,t) =0$  and $\displaystyle \lim_{t \to \infty}\phi(x,t)=\infty$ for every $x\in \Omega$.
\item 
The function $t \mapsto \frac{\phi(x, t)}t$ is $L$-almost increasing on $(0,\infty)$ with $L$ independent of $x$.
\end{itemize}
If $\phi\in\Phiw(\Omega)$ and  additionally $t \mapsto \phi(x, t)$  is convex and left-continuous for almost every $x$, then $\phi$ is a 
\textit{convex $\Phi$-function}, and we write $\phi \in \Phic(\Omega)$. 
If $\phi\in\Phiw(\Omega)$ and  additionally $t \mapsto \phi(x, t)$  is  convex and  continuous in the topology of $[0, \infty]$ for almost every $x$, then $\phi$ is a 
\textit{strong $\Phi$-function}, and we write $\phi \in \Phis(\Omega)$. 

\end{defn}

A function $\phi \in \Phi_c(\Omega)$ is called $N$-function if for almost every $x$ we have $\phi(x,t) \in (0, \infty)$ for all $t>0$, $\lim_{t\to 0^+}\frac{\phi(x,t)}{t}=0$ and
$\lim_{t\to \infty}\frac{\phi(x,t)}{t}=\infty$. $N$-functions are always continuous, since they are finite and convex, and thus they are strong $\Phi$-functions.

Two functions $\phi$ and $\psi$ are \textit{equivalent}, 
$\phi\simeq\psi$, if there exists $L\ge 1$ such that 
$\psi(x,\frac tL)\le \phi(x, t)\le \psi(x, Lt)$ for every $x \in \Omega$ and every $t>0$.
Equivalent $\Phi$-functions give rise to the same space with 
comparable norms. 

Two functions $\phi$ and $\psi$ are \textit{weakly equivalent}, 
$\phi\sim\psi$, if there exist $L\ge 1$  and $h \in L^1(\Omega)$ such that 
$\psi(x,t)\le \phi(x, Lt) + h(x)$ and $\phi(x, t)\le \psi(x, Lt) + h(x)$ for all $t \ge 0$ and almost all $x \in \Omega$.
Weakly equivalent $\Phi$-functions give rise to the same space with 
comparable norms.

By $\phi^{-1}(x,t)$ we mean a generalized inverse defined by
\begin{align*}
\phi^{-1}(x,t) := \inf\{\tau \in \R : \phi(x,\tau) \geq t\}.
\end{align*}
Functions $\phi$ and $\psi$ are equivalent if and only if $\phi^{-1}(x,t) \approx \psi^{-1}(x,t)$.

For $\phi \in \Phi_w(\Omega)$ we define a conjugate $\phi$-function $\phi^{\ast} \in \Phi_w(\Omega)$ by
\begin{align*}
\phi^\ast(x,t) := \sup_{s>0} st - \phi(x,s).
\end{align*}

We collect two results how the generalized inverse behaves. For the proofs see
Lemma 2.3.3  and Theorem 2.4.8 in \cite{HarH_book}.

\begin{lem}\label{lem:composition}\label{lem:conjugate-inverse}
Let $\phi \in \Phiw(\Omega)$. Then
\begin{itemize}
\item[(a)]  $\phi(x, \phi^{-1}(x, t)) =t$ when $\phi \in \Phis(\Omega)$,
\item[(b)]  $\phi^{-1}(x,t) \left (\phi^{\ast}\right )^{-1}(x,t) \approx t$.
\end{itemize} 
\end{lem}

We define the following conditions. They guarantee boundedness of maximal operators and density of smooth functions. 
\begin{defn}
We say that $\phi:\Omega\times [0,\infty)\to [0,\infty)$ satisfies 
\begin{itemize}
\item[(aInc)$_p$] \label{def:aInc}
if
$t \mapsto \frac{\phi(x,t)}{t^{q}}$ is $L_p$-almost 
increasing in $(0,\infty)$ for some $L_p\ge 1$ and a.e.\ $x\in\Omega$.
\item[(aDec)$_q$] \label{def:aDec}
if
$t \mapsto \frac{\phi(x,t)}{t^{q}}$ is $L_q$-almost 
decreasing in $(0,\infty)$ for some $L_q\ge 1$ and a.e.\ $x\in\Omega$.
\item[(A0)]\label{def:a0}
if there exists $\beta \in(0, 1]$ such that $ \beta \le \phi^{-1}(x,
1) \le \frac1{\beta}$ for almost every $x \in \Omega$. 

\item[(A1)] \label{def:a1}
if there exists $\beta\in (0,1)$ such that
\[
\beta \phi^{-1}(x, t) \le \phi^{-1} (y, t)
\]
for every $t\in [1,\frac 1{|Q|}]$, almost every $x,y\in Q \cap \Omega$ and 
every cube $Q$ with $|Q|\le 1$.

\item[(A2)]\label{def:a2}
if there exists $\phi_\infty \in \Phiw$, $h \in L^1(\Omega)\cap L^\infty(\Omega)$, $\beta \in(0, 1]$ and $s>0$ such that
\[
\phi(x, \beta t) \le \phi_\infty(t) + h(x) \quad \text{ and } \quad \phi_\infty(\beta t) \le \phi(x, t) + h(x) 
\] 
for almost every $x \in \Omega$ when $\phi_\infty(t) \in [0, s]$ and $\phi(x, t) \in [0,s]$, respectively.
\end{itemize} 
\end{defn}

We say $\phi$ satisfies \ainc{} if it satisfies \ainc{p} for some $p>1$ and similarly \adec{} if it satisfies \adec{q} for some $q<\infty$. Conditions \ainc{} and \adec{} correspond to the $\nabla_2$ and $\Delta_2$ conditions respectively from the classical Orlicz space theory.

\begin{rem}\label{rem:assumptions}
We use the following properties without mentioning them explicitly.
\begin{itemize}
\item[(a)] The conditions \azero{} - \atwo{}, \ainc{p} and \adec{q} are invariant under equivalence ($\simeq$).
\item[(b)] If conditions \azero{} - \atwo{}, \ainc{p} and \adec{q} hold for $\phi$, then they hold also for $\phi^\ast$ (see \cite{HarH_book}).
\item[(c)] For any $p,q >0$ a generalized $\Phi$-functions satisfies \ainc{p} and \adec{q} if and only if $\phi^{-1}$ satisfies \ainc{1/q} and \adec{1/p} respectively.
\item[(d)] If $\phi$ satisfies \adec{q_1}, then it satisfies \adec{q_2} for all $q_2 >q_1$.
\end{itemize}

\end{rem}

The generalized Orlicz space $L^{\phi}(\Omega)$ comprises of measurable functions $f$ that satisfy
\begin{align*}
\int_{\Omega} \phi(x, \lambda f(x)) \, dx < \infty
\end{align*}
for some $\lambda >0$. $L^\phi(\Omega)$ is a (quasi) Banach space when equipped with a (quasi)norm
\begin{align*}
\|f\|_{L^{\phi}(\Omega)} := \inf \left \{ \lambda >0 : \int_{\Omega} \phi\left (x, \frac{f(x)}{\lambda}\right ) \, dx \leq 1 \right \}.
\end{align*}
If the set $\Omega$ is clear from the context, we abbreviate $\|f\|_{L^\phi(\Omega)}$ as $\|f\|_{\phi}$. A comprehensive presentation of generalized Orlicz spaces can be found in \cite{HarH_book}.

The conjugate $\Phi$-function generates the associate space of $L^{\phi}(\Omega)$ as the following Lemma shows. 

\begin{lem}[Norm conjugate formula, Theorem 3.4.6 in \cite{HarH_book}]\label{lem:norm-conjugate} If $\phi \in \Phi(\Omega)$, then for all measurable $f$
\begin{align*}
\|f\|_{L^{\phi}(\Omega)} \approx \sup_{\|g\|_{L^{\phi^\ast}(\Omega)\leq 1}} \int_{\Omega} |f(x)g(x)| \ dx. 
\end{align*}
\end{lem}

We extend the classical H\"older's inequality to generalized Orlicz spaces as follows.
\begin{lem}[Lemma 3.2.11 in \cite{HarH_book}]\label{lem:holder}
Let $\phi \in \Phi_w(\Omega)$, $f \in L^{\phi}(\Omega)$ and $g \in L^{\phi^{\ast}}(\Omega)$. Then
\begin{align*}
\int_\Omega f(x)g(x) \, dx \leq 2\|f\|_{L^{\phi}(\Omega)} \|g\|_{L^{\phi^{\ast}}(\Omega)}.
\end{align*}
Here the constant cannot be lower than 2 in general.
\end{lem}

We can also have a $\phi$-norm on the left-handside.
\begin{lem}[Generalized H\"older's inequality]\label{lem:gen-holder} 
Let $\phi_i \in \Phi_w(\Omega)$ for $i=1,2,3$. If for all $t\geq 0$ and almost every $x \in \Omega$ we have
\begin{align*}
\phi_1^{-1}(x,t) \phi_{2}^{-1}(x,t) \leq \phi_3^{-1}(x,t),
\end{align*} 
then 
\begin{align*}
\|fg\|_{L^{\phi_3}(\Omega)} \lesssim 2 \|f\|_{L^{\phi_1}(\Omega)} \|g\|_{L^{\phi_2}(\Omega)}
\end{align*}
for every $f \in L^{\phi_1}(\Omega)$ and $g \in L^{\phi_2}(\Omega)$. If the $\Phi$-functions are convex, then ''$\lesssim$'' can be replaced with ''$\leq$''.
\end{lem}

\begin{proof}
Let us first suppose that $\phi_i \in \Phi_c(\Omega)$. We start the proof with an associated Young's inequality. Directly from the definitions we deduce $\phi_{1}(x, \phi_{i}^{-1}(x,t)) \leq t \leq \phi_{i}^{-1}(x,\phi_i(x,t))$. We have two cases: either $\phi_1(x,s) \leq \phi_2(x,t)$ or $\phi_1(x,s) > \phi_2(x,t)$. In the first situation
\begin{align*}
st &\leq \phi_1^{-1}(x,\phi_1(x,s)) \phi_2^{-1}(x, \phi_2(x,t)) \leq \phi_1^{-1}(x, \phi_2(x,t))\phi_2^{-1}(x,\phi_2(x,t))\\
&\leq \phi_3^{-1}(x,\phi_2(x,t)).
\end{align*} 
Applying $\phi_3(x,\cdot)$ to both sides, we have
\begin{align*}
\phi_3(x,st) \leq \phi_2(x,t).
\end{align*}
If on the other hand we are in situation 2, similar estimates yield $\phi_3(x,st) \leq \phi_1(x,s)$, so we have the Young's inequality
\begin{align*}
\phi_3(x,st) \leq \max\{\phi_1(x,s), \phi_2(x,t)\} \leq \phi_1(x,s) + \phi_2(x,t).
\end{align*}

Now we can prove the generalized H\"older's inequality. Let $\varepsilon >0$ and without loss of generality assume that $\|f\|_{\phi_1}=\|g\|_{\phi_2}=1$. Now the convexity of $\phi_3$ and the previous Young's inequality yield
\begin{align*}
\int_{\Omega}\phi_3\left (x, \dfrac{f(x)g(x)}{2(1+\varepsilon)^2}\right ) \, dx \leq  \dfrac{1}{2} \int_{\Omega} \phi_1 \left (x, \dfrac{f(x)}{1+\varepsilon}\right ) \, dx + \dfrac{1}{2}\int_{\Omega} \phi_2 \left (x, \dfrac{g(x)}{1+\varepsilon}\right ) \, dx \leq 1.
\end{align*}
In other words, $\|fg\|_{\phi_3} \leq 2(1+\varepsilon)^2$ and thus the claim follows by letting $\varepsilon \to 0$.

Suppose now we are not in the convex regime, that is $\phi_i \in \Phi_w(\Omega)$. There exist $\psi_i \in \Phi_c(\Omega)$ such that $\phi_i \simeq \psi_i$ \cite[Lemma 2.2.1]{HarH_book}. Now using the Young's inequality for convex functions $\psi_i$ and  equivalence, we end up with 
\begin{align*}
\phi_3\left (x, \frac{st}{L^3}\right ) \leq \psi_3\left (x, \frac{st}{L^2} \right ) \leq \psi_1\left (x,\frac{s}{L}\right ) + \psi_2\left (x,\frac{t}{L}\right ) \leq \phi_1(x,s) + \phi_2(x,t)
\end{align*}
where $L$ is the largest equivalence constant of the three $\Phi$-functions. Denoting the constant of \ainc{1} by $a$, we have
\begin{align*}
\int_{\Omega}\phi_3\left (x, \dfrac{f(x)g(x)}{2aL^3(1+\varepsilon)^2}\right ) \, dx \leq  \dfrac{1}{2} \int_{\Omega} \phi_1 \left (x, \dfrac{f(x)}{1+\varepsilon}\right ) \, dx + \dfrac{1}{2}\int_{\Omega} \phi_2 \left (x, \dfrac{g(x)}{1+\varepsilon}\right ) \, dx \leq 1.
\end{align*}
Thus $\|fg\|_{\phi_3} \leq 2 aL^3(1+\varepsilon)^2 \to 2aL^3$ as $\varepsilon \to 0$ and inequality is proven.
\end{proof}

We have the following rather crude but sometimes suitable estimate between norms and and integrals.
\begin{lem}\label{lem:crude}
Assume that $\phi \in \Phi_w(\Omega)$ satisfies \ainc{p} and \adec{q} for any $1 \leq p \leq q < \infty$. Then
\begin{align*}
\|f\|_{L^{\phi}(\Omega)} \lesssim \max \left \{ \left (\int_{\Omega} \phi(x,f(x)) \, dx \right )^{\frac{1}{p}}, \left (\int_{\Omega} \phi(x,f(x)) \, dx \right )^{\frac{1}{q}} \right \}.
\end{align*}
\end{lem}

The explicit norms of functions are often difficult to calculate, but under suitable assumptions this can be done for characteristic functions of simple sets such as cubes.
\begin{lem}[Proposition 4.4.11 in \cite{HarH_book}]\label{lem:norm-of-ball}
Let $\phi \in \Phi_w(\Omega)$ satisfy \azero{}, \aone{} and \atwo{}. Then for every cube $Q \subset \Omega$ we have
\begin{align*}
\|\chi_Q\|_{L^{\phi}(\Omega)} \|\chi_Q\|_{L^{\phi^\ast}(\Omega)} \approx |Q|.
\end{align*}
Here the implicit constant is independent of the cube $Q$.
\end{lem}

The previous lemma combined with boundedness of averaging operator yields the following result (see \cite[Proposition 4.4.11]{HarH_book}).
\begin{lem}\label{lem:measure-scaling}
Let $\phi \in \Phi_w(\Omega)$ satisfy \azero{}, \aone{} and \atwo{}. Then for every cube $Q \subset \Omega$ we have
\begin{align*}
\|\chi_{Q}\|_{L^{\phi}(\Omega)} \approx \dfrac{1}{\fint_{Q}\phi^{-1}\left (x, \frac{1}{|Q|}\right )\,dx}.
\end{align*}
\end{lem}

The following assumptions will always be repeated when discussing fractional maximal operator or Riesz potential. Equivalently, we could assume that $\psi$ satisfies \azero{}, \aone{}, \atwo{}, \ainc{q} and \adec{1/(r-\alpha/n)}.
\begin{assumptions}\label{assu:fractional}
Let $1 < p \leq q < \infty$, $\frac{\alpha}{n} = \frac{1}{p} - \frac{1}{q}$ and $\phi, \psi \in \Phi_w(\Rn)$ be such that
\begin{equation}\label{eq:phi-psi}
\phi^{-1}(x,t) \approx  t^{\frac{\alpha}{n}}\psi^{-1}(x,t).
\end{equation}
Let $r \in  (\frac{\alpha}{n}, \frac{1}{p}]$. Assume that $\phi$ satisfies \azero{}, \aone{}, \atwo{}, \ainc{p} and \adec{1/r}. 
\end{assumptions}

\subsection{Other function spaces and operators}

Next we introduce other function spaces appearing in this paper. Functions of bounded mean oscillation, that is functions $b \in L^{1}_\loc(\Rn)$ satisfying
\begin{align*}
\dfrac{1}{|Q|} \int_{Q} |b(x) - b_Q| \, dx < \infty, \quad \text{where} \quad b_Q = \dfrac{1}{|Q|} \int_Q b(x) \, dx
\end{align*}
for every cube $Q \subset \Rn$, play an essential role in the theory of commutators. They form a function space $\BMO(\Rn)$ when equipped with a seminorm
\begin{align*}
\|b\|_{\ast} = \sup_{Q \subset \Rn} \dfrac{1}{|Q|} \int_{Q}|b(x) - b_Q| \, dx.
\end{align*}

Additionally we define the negative part of a function $b^{-}(x):= -\min\{b(x),0\}$ and similarly the positive part $b^{+}(x) = \max\{b(x),0\}$. It immediately follows that for every $x$ we have $|b(x)|-b(x) = 2b^{-}(x)$.

The space of compactly supported and smooth functions is denoted as $C^{\infty}_0(\Omega)$. This space is dense in $L^{\phi}(\Rn)$ if $\phi$ satisfies \azero{} and \adec{} \cite[Theorem 3.7.15]{HarH_book}.

Lastly, we define various operators investigated in this paper. All of the maximal operators are derived from non-centred Hardy--Littlewood maximal function
\begin{align*}
M f(x) :=  \sup_{Q \ni x} \dfrac{1}{|Q|} \int_{Q} |f(y)| \, dy.
\end{align*}
This is comparable with the centred variant
\begin{align*}
M^{c} f(x) :=  \sup_{Q(x,r)} \dfrac{1}{|Q(x,r)|} \int_{Q(x,r)} |f(y)| \, dy
\end{align*}
as they enjoy pointwise inequalities $M^cf(x) \leq M f(x) \leq 2^{n}M^c f(x)$. We can also have a restricted maximal function in each cube $Q_0$ as
\begin{align*}
M_{Q_0}f(x) := \sup_{\substack{Q \ni x \\Q \subset Q_0}} \dfrac{1}{|Q|}\int_{Q}|f(y)| \, dy
\end{align*}
We also define the sharp maximal function as
\begin{align*}
M^{\sharp}f(x) = \sup_{Q \ni x}\dfrac{1}{|Q|}\int_{Q} |f(x) - f_{Q}| \, dx.
\end{align*}

If a measurable function $b$ is given, we define the maximal commutator as
\begin{align*}
M_b f(x) := \sup_{Q \ni x} \dfrac{1}{|Q|} \int_{Q}|b(x)-b(y)| |f(y)| \, dy.
\end{align*}
On the other hand, if a measurable function $b$ is given, the commutator of a maximal function is defined as
\begin{align*}
[M, b] f(x) := M(bf)(x) - b(x) Mf(x).
\end{align*}

For $\alpha \in (0, n)$, the main focus of this article is the fractional maximal function
\begin{align*}
M_\alpha f(x) := \sup_{Q \ni x}\dfrac{1}{|Q|^{1-\frac{\alpha}{n}}} \int_{Q} |f(y)| \, dy
\end{align*}
and its variations similar to previously given to the Hardy--Littlewood maximal function.

Closely related to fractional maximal function is the Riesz potential of a function $f$
\begin{align*}
I_\alpha f(x) := \int_{\Rn} \dfrac{f(y)}{|x-y|^{n-\alpha}} \, dy.
\end{align*}
For the Riesz potential we also have the variants
\begin{align*}
[I_\alpha, b]f(x) = I_\alpha(bf)(x) - b(x)I_\alpha f(x) \quad \text{and} \quad I_{b,\alpha}f(x) = \int_{\Rn}\dfrac{|b(x)-b(y)|}{|x-y|^{n-\alpha}}f(y) \, dy.
\end{align*}
Note that $I_{b,\alpha}$ is often denoted in the literature as $|I_\alpha,b|$ or $|b,I_\alpha|$, but to unify notation with maximal functions we use a different notation.

Boundedness of Hardy--Littlewood maximal operator in generalized Orlicz spaces with the current framework was proven by H\"ast\"o \cite{Has15}. In particular, if $\phi\in \Phi_w(\Omega)$ satisfies \azero{}, \aone{}, \atwo{} and \ainc{}, we have
\begin{align}\label{eq:M-phi-bdd}
\|Mf\|_{\phi} \leq C_M \|f\|_{\phi}
\end{align}
for all $f \in L^{\phi}(\Omega)$.

\section{Boundedness of the fractional maximal operator}

In this section we prove boundedness properties of the fractional maximal operator. The result follows also from pointwise inequality $M_\alpha f(x) \lesssim I_\alpha f(x)$ and \cite{HarH_Riesz} or from extrapolation arguments \cite{extrapolation} but here we give a direct proof for the result and simplify the presentation of function spaces.
To obtain this, we exploit the fact, that $\phi$ can be regularized to a function $\tilde \phi$, while preserving the function spaces and having equivalent norms. 
This construction was developed in \cite{HarH_Riesz} but we demonstrate the steps involved and show that our framework, Assumptions \ref{assu:fractional}, is the same. 
The main reason to consider regularized $\Phi$-function is to extend range of $t$ in \aone{} to $\left [0, \frac{1}{|Q|}\right ]$ (see Proposition 4.5 in \cite{HarH_Riesz}). 
Note that we do not assume that $\phi$ is an $N$-function as was done in \cite{HarH_Riesz}, since the critical result, $\phi^{-1}(x,t)\left (\phi^{\ast}\right )^{-1}(x,t) \approx t$, is now known to hold for weak $\Phi$-functions also.
\begin{enumerate}
\item We start with a generalized Orlicz function $\phi$ satisfying \azero{}, \aone{}, \atwo{}, \ainc{p} and \adec{1/r}, where $r \in (\frac{\alpha}{n}, \frac{1}{p}]$ and $p>1$.
\item Choose $\phi_1 \in \Phi_s(\Rn)$ such that $\phi_1(x,1)=\phi_1^{-1}(x,1)=1$ and $\phi_1 \simeq \phi$ \cite[Lemma 3.7.3]{HarH_book}.
\item Define $\phi_2(x,t) = \max \{\phi_1(x,t), 2t-1\}$. We have $\phi_2 \simeq \phi_1\simeq \phi$ and therefore $\phi_2$ satisfies \azero{}, \aone{}, \atwo{}, \ainc{p} and \adec{1/r}.
\item Define
\begin{align*}
\tilde \phi(x,t) = 
\begin{cases}
(\phi_2)_{\infty}(t) &\text{ when } t <1, \\
2 \phi_2(x,t) -1 &\text{ when } t \geq 1,
\end{cases}
\end{align*}
where $(\phi_2)_\infty(t) = \limsup_{|x| \to \infty}\phi_2(x,t)$.
\end{enumerate}
In \cite[Proposition 4.2]{HarH_Riesz}  it was shown that $L^{\phi}(\Rn) = L^{\tilde \phi}(\Rn)$ with comparable norms and thus $\phi \sim \tilde \phi$.

\cite[Proposition 4.5]{HarH_Riesz} shows that the if $\phi$ satisfies \azero{}, \aone{} and \atwo{}, then so does $\tilde \phi$ (condition \azero{} follows automatically, since $\tilde \phi(x,1)=1$ by construction). We add to this by noting that regularization preserves also \ainc{} and \adec{}: for $t<s<1$ we have
\begin{align*}
\dfrac{\tilde \phi(x,t)}{t^p} = \limsup_{|x| \to \infty} \dfrac{\phi_2(x,t)}{t^{p}} \lesssim \limsup_{|x| \to \infty} \dfrac{\phi_2(x,s)}{s^{p}} = \dfrac{\tilde \phi(x,s)}{s^{p}},
\end{align*} 
for $1<t<s$ we have
\begin{align*}
\dfrac{\tilde \phi(x,t)}{t^{p}} &= \dfrac{2 \phi_2(x,t)-1}{t^{p}} \lesssim \dfrac{2\phi_2(x,s)}{s^p} \leq 2\dfrac{2\phi_2(x,s)-1}{s^p} = 2 \dfrac{2 \tilde \phi(x,s)}{s^{p}}
\end{align*}
and finally for $t<1<s$ \ainc{} follows since $\tilde \phi$ is increasing. \adec{} is proven similarly.

There is still the question of relation between regularizing and satisfying Assumptions \ref{assu:fractional}. Since $\tilde{\phi}$ satisfies these assumptions, then by \cite[Lemmas 5.2.3 and 3.7.3]{HarH_book} these exists $\psi_0 \in \Phi_s(\Omega)$  satisfying \ainc{q} and \adec{1/(r-\alpha/n)} such that $\tilde \phi^{-1}(x,t) \approx t^{\frac{\alpha}{n}} \psi_{0}^{-1}(x,t)$.
Let us show that $\psi_0 \sim \psi$ when $\phi^{-1}(x,t) \approx t^{\frac{\alpha}{n}} \psi^{-1}(x,t)$. Equivalently this means that $\psi_0^{-1}(x,t) \approx \psi^{-1}(x, t+h(x))$, where $h \in L^{1}(\Rn)$.  Since \adec{n/\alpha} of $\phi$ implies \ainc{\alpha/n} of $\phi^{-1}$ and we know that $\phi \sim \tilde \phi$, we estimate
\begin{align*}
\psi_0^{-1}(x,t) &\approx \dfrac{\tilde \phi^{-1}(x,t)}{t^{\alpha/n}} \lesssim \dfrac{\tilde \phi^{-1}(x, t + h(x))}{[t + h(x)]^{\alpha/n}} \approx \dfrac{\phi^{-1}(x, t + 2h(x))}{[t + h(x)]^{\alpha/n}} \lesssim \dfrac{\phi^{-1}(x, t + 2h(x))}{[t +2h(x)]^{\alpha/n}} \\
&\approx \psi^{-1}(x, t + 2h(x)),
\end{align*}
i.e. $\psi_0 \sim \psi$. From now on, we denote $\tilde \psi := \psi_0$ since it possesses all the relevant properties of a regularized function: weak equivalence, \azero{} and \aone{} in range $t \in \left [0, \frac{1}{|Q|}\right ]$.

\begin{rem}\label{rem:Riesz-space}
In the paper of Harjulehto and H\"ast\"o, the target space of $I_\alpha$ was constructed in a slightly different way than in our presentation. Let us show that they are in fact equivalent: 
We show that under suitable structural conditions on $\phi$, the Riesz potential maps functions from $L^{\phi}(\Rn) \to L^{\psi}(\Rn)$, where $\phi$ and $\psi$ satisfy the relation $\phi^{-1}(x,t) \approx t^{\frac{\alpha}{n}} \psi^{-1}(x,t)$. In \cite{HarH_Riesz} the function target space is $L^{\phi_\alpha^{\sharp}}(\Rn)$, where $\phi^{\sharp}$ is a convex $\Phi$-function equivalent to $\phi(x, \lambda^{-1}(x,t))$ and $\lambda(x,t) := t \phi(x,t)^{-\frac{\alpha}{n}}$. Recall that equivalence of $\Phi$-functions is an equivalent condition to comparability of their inverses. The claim is invariant under equivalence of $\Phi$-functions so we can assume that $\phi, \psi \in \Phi_s(\Omega)$ and therefore by Lemma \ref{lem:composition} (a)
\begin{align*}
\big(\phi_\alpha^{\sharp}\big)^{-1}(x,t) \approx \lambda(x,\phi^{-1}(x,t)) = \phi^{-1}(x,t) \ \phi(x, \phi^{-1}(x,t))^{-\frac{\alpha}{n}} = t^{-\frac{\alpha}{n}}\phi^{-1}(x,t) \approx \psi^{-1}(x,t).
\end{align*}
Thus $\phi_{\alpha}^{\sharp} \simeq \psi$ and they generate the same function spaces with comparable norms and we resume to denote $\psi$ for the function generating the target space. 
\end{rem}

We collect these observations to the following Lemma. Note that the now \aone{} is valid for cubes of all size, not just those with $|Q|\leq 1$.
\begin{lem}\label{lem:reg-psi}
Let $\phi$ satisfy Assumptions \ref{assu:fractional}. Then there exist $\tilde \phi, \psi, \tilde \psi \in \Phi_s(\Omega)$ such that 
\begin{equation*}
\phi \sim \tilde \phi, \quad \psi \sim \tilde \psi,\quad \phi^{-1}(x,t) \approx t^{\frac{\alpha}{n}} \psi^{-1}(x,t), \quad \tilde \phi^{-1}(x,t) \approx t^{\frac{\alpha}{n}} \tilde \psi^{-1}(x,t)
\end{equation*}
and 
\begin{equation*}
\tilde{\psi}^{-1}(x,t) \lesssim \tilde{\psi}^{-1}(y,t)
\end{equation*}
for every cube $Q$, almost every $x, y \in Q$ and every $t \in \left [0, \frac{1}{|Q|}\right ]$.
\end{lem}

The first main result is a direct proof, without extrapolation arguments, that the fractional maximal operator is bounded. 

\begin{thm}\label{thm:M_alpha-bounded}
Let $1 < p \leq q < \infty$, $\frac{\alpha}{n} = \frac{1}{p} - \frac{1}{q}$ and $\phi, \psi \in \Phi_w(\Rn)$ be such that
\begin{equation*}
\phi^{-1}(x,t) \approx  t^{\frac{\alpha}{n}}\psi^{-1}(x,t).
\end{equation*}
Let $r \in  (\frac{\alpha}{n}, \frac{1}{p}]$. Assume that $\phi$ satisfies \azero{}, \aone{}, \atwo{}, \ainc{p} and \adec{1/r}. Then the fractional maximal operator $M_\alpha: L^{\phi}(\Rn) \to L^{\psi}(\Rn)$ is bounded. 
\end{thm}

\begin{proof}
Recall the centred fractional maximal function
\begin{align*}
 M^c_\alpha f(x) := \sup_{r>0} \dfrac{1}{|Q(x,r)|^{1-\alpha/n}} \int_{Q(x,r)} |f(y)| \, dy
\end{align*}
and recall that $M^c_\alpha f(x) \approx  M_\alpha f(x)$. Thus, we can prove the boundedness for the centred fractional maximal operator.

Let $\Phi$-functions $\tilde \phi$ and $\tilde \psi$ be as in Lemma \ref{lem:reg-psi}.
Let $Q(x,r)$ be any cube and split a function $f$ as
\begin{equation*}
f= f_1 + f_2, \quad f_1(y) = f(y)\chi_{Q(x,r)}(y), \quad f_2(y)=f(y)\chi_{Q(x,r)^{c}}(y).
\end{equation*}
By sublinearity of the maximal operator we have $M^c_\alpha f(x) \leq M^c_\alpha f_1(x) + M_\alpha^c f_2(x)$. We start by estimating the last term
\begin{align*}
M_\alpha^c f_2(x) &\lesssim \sup_{t>0} \dfrac{1}{|Q(x,t)|^{1-\alpha/n}} \int_{Q(x,t) \cap Q(x,r)^{c}} |f(y)| \, dy \\
&\lesssim \sup_{t >r} \dfrac{1}{|Q(x,t)|^{1-\alpha/n}} \int_{Q(x,t)} |f(y)| \, dy.
\end{align*}
We continue the estimates with H\"older's inequality, Lemma \ref{lem:norm-of-ball} and Lemma \ref{lem:composition} (b)
\begin{align*}
M^c_\alpha f_2(x)&\lesssim \sup_{t >r} \|f\|_{\tilde \phi} |Q(x,t)|^{\alpha/n-1} \|\chi_{Q(x,t)}\|_{\tilde \phi^\ast}\\
&\lesssim \sup_{t >r} \|f\|_{\tilde \phi} |Q(x,t)|^{\alpha/n-1} \left (\fint_{Q(x,t)} \left (\tilde \phi^{\ast}\right )^{-1}(y,|Q(x,t)|^{-1})\, dy\right )^{-1} \\
&\lesssim \sup_{t >r} \|f\|_{\tilde \phi} |Q(x,t)|^{\alpha/n-1} \left (\fint_{Q(x,t)} \dfrac{1}{|Q(x,t)|}  \dfrac{1}{\tilde \phi^{-1}(y,|Q(x,t)|^{-1})} \, dy\right )^{-1}.
\end{align*}
Using convexity of the function $z \mapsto \frac{1}{z}$ with Jensen's inequality and Lemma \ref{lem:reg-psi} we continue
\begin{align*}
M^c_\alpha f_2(x) &\lesssim \sup_{t >r}  \|f\|_{\tilde \phi} |Q(x,t)|^{\alpha/n} \fint_{Q(x,t)} \tilde \phi^{-1}(y,|Q(x,t)|^{-1}) \, dy \\
&\lesssim \|f\|_{\tilde \phi} \sup_{t>r} t^{\alpha}  \fint_{Q(x,t)} \tilde \phi^{-1}(y,t^{-n}) \, dy \\
&\leq\|f\|_{\tilde\phi} \sup_{t>r}   \tilde\psi^{-1}(x,t^{-n}) \leq \|f\|_{\tilde\phi}   \tilde\psi^{-1}(x,r^{-n}).
\end{align*}
For $f_1$ we have $M^c_\alpha f_1(x) \lesssim r^{\alpha} M^cf(x) \approx r^{\alpha}Mf(x)$ (see for example \cite[Lemma. (b)]{Hed72}), so returning to non-centred maximal functions we see
\begin{align*}
|M_\alpha f(x)| &\lesssim r^{\alpha} Mf(x) +\|f\|_{\tilde \phi} \tilde\psi^{-1}(x,r^{-n}) \\
&\lesssim Mf(x) \dfrac{\tilde \psi^{-1}(x,r^{-n})}{\tilde \phi^{-1}(x,r^{-n})} + \|f\|_{\phi}  \tilde\psi^{-1}(x,r^{-n}).
\end{align*}
Recall that $\tilde \phi \in \Phi_s(\Rn)$ so it is continuous and surjective from $[0,\infty)$ to $[0, \infty)$. Thus we can choose $r$ in such a way that $\tilde\phi^{-1}(x,r^{-n}) = \frac{Mf(x)}{C_M\|f\|_{\tilde\phi}}$, where $C_M$ is the constant in \eqref{eq:M-phi-bdd}, and get
\begin{align*}
|M_\alpha f(x)| \lesssim \|f\|_{\tilde\phi}  \tilde\psi^{-1}\left (x, \tilde \phi\left (x, \dfrac{Mf(x)}{C_M \|f\|_{\tilde\phi}}\right )\right ).
\end{align*}
Since $\tilde \psi$ is increasing and a strong $\Phi$-function, which allows us to use Lemma \ref{lem:composition} (a), it follows that
\begin{align*}
\tilde\psi\left (x, \dfrac{|M_\alpha f(x)|}{C_M\|f\|_{\tilde\phi}} \right ) &\lesssim \tilde\psi\left (x,\tilde\psi^{-1}\left (x,\tilde\phi\left (x, \frac{Mf(x)}{C_M\|f\|_{\tilde\phi}}\right )\right )\right ) = \tilde \phi\left (x, \dfrac{Mf(x)}{C_M \|f\|_{\tilde \phi}}\right ).
\end{align*} 
Integrating both sides we have
\begin{align*}
\int_{\Rn}\tilde\psi\left (x, \dfrac{|M_\alpha f(x)|}{C_M\|f\|_{\tilde\phi}} \right ) \, dx \lesssim \int_{\Rn} \tilde \phi\left (x, \dfrac{Mf(x)}{C_M \|f\|_{\tilde \phi}}\right ) \, dx \leq 1.
\end{align*}
In other words $\|M_\alpha f\|_{\tilde \psi} \lesssim \|f\|_{\tilde \phi}$. But $\tilde \phi$ and $\tilde \psi$ are weakly equivalent to $\phi$ and $\psi$ so they have comparable norms, respectively, and all in all
\begin{equation*}
\|M_\alpha f\|_{L^{\psi}(\Rn)} \lesssim \|f\|_{L^{\phi}(\Rn)}. \qedhere
\end{equation*}
\end{proof}

\begin{rem}
Contrary to regular Orlicz case, we do not know if the assumption $\phi^{-1}(x,t) \approx t^{\frac{\alpha}{n}}\psi^{-1}(x,t)$ is necessary. To best of our knowledge, this is an open question already in the variable exponent case. Following the proof of Orlicz case in \cite{orlicz1}, the problem arises as $\|\chi_Q\|_{\phi}$ cannot always be estimated pointwise as $\left (\phi^{-1}(x, |B|^{-1})\right )^{-1}$ but only in the integral from as in Lemma \ref{lem:norm-of-ball}.
\end{rem}

\section{Commutators of Riesz potential}

It turns out that the boundedness of $[I_\alpha,b]$ can be tackled with the sharp maximal function. Our methods are based on \cite{p(x)-riesz}. There is no pointwise inequality between between a function $f$ and the sharp maximal function $M^{\sharp}f$ but a norm estimate can be achieved as is done in Lemma \ref{lem:sharp-norm}. However, to prove this we need to investigate smooth functions with compact support and vanishing integral. We denote
\begin{align*}
C^{\infty}_{0,0}(\Omega) = \left \{ f \in C^{\infty}_0(\Omega) : \int_{\Omega} f(x) \, dx = 0\right \}.
\end{align*}
For bounded domains $\Omega$ we define $L^{\phi}_0(\Omega)$ analogously but for for domains with $|\Omega|=\infty$ we define $L^{\phi}_0(\Omega) = L^{\phi}(\Rn)$.
\begin{lem}
Assume that $\phi \in \Phi_w(\Omega)$ satisfies \azero{} and \adec{q}. If $|\Omega| < \infty$ or $\phi$ satisfies \ainc{p} with $p>1$, then $C^{\infty}_{0,0}(\Omega)$ is dense in $L^{\phi}_0(\Omega)$.
\end{lem}

\begin{proof}
Let us first investigate the case of a bounded domain. Let $\eta \in C^{\infty}_0(\Omega)$ satisfy $\int_{\Omega}\eta(x) \, dx = 1$. Since $\phi$ satisfies \azero{} and \adec{}, compactly supported smooth functions are dense in $L^{\phi}_0(\Rn) \subset L^{\phi}(\Omega)$, and therefore we have a sequence $(\tilde f_k) \subset C_0^{\infty}(\Omega)$ such that $\tilde f_k \to f$ in $L^{\phi}(\Omega)$. H\"older's inequality implies that $\|f-\tilde{f_k}\|_{L^{1}(\Omega)} \lesssim \|f-\tilde{f_k}\|_{\phi} \|\chi_\Omega\|_{\phi^{\ast}}$ so $(\tilde f_k)$ converges to $f$ also in $L^{1}(\Omega)$ and
\begin{align*}
\int_{\Omega} \tilde f_k(x) \, dx \to \int_{\Omega} f(x) \, dx = 0.
\end{align*}
Now we define $f_k := \tilde{f_k} - \eta \int_{\Omega} \tilde f_k(x)\, dx$ and notice that $f_k \in C^{\infty}_{0,0}(\Omega)$ and
\begin{align*}
\|f-f_k\|_{\phi} \leq \|f-f_k\|_{\phi} + \|\eta\|_{\phi} \left |\int_{\Omega} \tilde{f_k}(x) \, dx \right |.
\end{align*}
Letting $k \to \infty$, we see that $\|f-f_k\|_{\phi}\to 0$ and the proof for bounded $\Omega$ is complete.

Suppose then that $|\Omega| = \infty$ and $\phi$ satisfies \ainc{}. Let $\{\Omega_j\}$ be an increasing sequence of bounded domains satisfying 
\begin{align*}
\Omega_j \Subset \Omega \quad \text{with} \quad \bigcup_{j=1}^{\infty}\Omega_j = \Omega \quad \text{and} \quad |\Omega_j| \geq 1.
\end{align*}
For each $\Omega_j$, associate a non-negative function $\eta_j \in C^{\infty}_0(\Omega)$ satisfying $\int_{\Omega_j} \eta_j(x) \, dx  =1$ and $\eta_j(x) \lesssim |\Omega_j|^{-1}\chi_{\Omega_j}$. This can be done by mollifying functions $|\Omega_j|^{-1}\chi_{\Omega_j}$ for example. From the pointwise inequality and Lemma \ref{lem:crude} we get
\begin{align}\label{eq:vanishing1}
\|\eta_j\|_{\phi} \lesssim |\Omega_j|^{-1}\|\chi_{\Omega_j}\|_{\phi} \lesssim \max\{ |\Omega_j|^{-1 + \frac{1}{p}}, |\Omega_j|^{-1+\frac{1}{q}}\} \to 0,
\end{align}
when $j \to 0$. By density of $C^{\infty}_0(\Omega)$ there exists functions $(\tilde{f_k}) \subset C^{\infty}_0(\Omega)$ such that $\tilde f_k \to f$ in $L^{\phi}(\Omega)$ and $\|\tilde{f}_k\|_{\phi} \leq \|f\|_{\phi} +1$. Similarly to the case of bounded domain, we define $f_k := \tilde{f_k} - \eta_{j_k} \int_{\Omega_k} \tilde f_k(x)\, dx$, where $j_k$ is an increasing sequence chosen later. By construction 
\begin{align*}
\int_{\Omega} f_k(x) \, dx = \int_{\Omega} \tilde f_k(x) \, dx - \int_{\Omega_{j_k}} \eta_{j_k} (x) \, dx \int_{\Omega} \tilde f_k(x) \, dx = 0,
\end{align*}
so $(f_k) \subset C^{\infty}_{0,0}(\Omega)$. H\"older's inequality for the integral term inside the norm yields
\begin{align*}
\|f-f_k\|_{\phi} &\lesssim \|f-\tilde{f_k}\|_{\phi} + \|\eta_{j_k}\|_{\phi}\|\chi_{\Omega_{k}}\|_{\phi^{\ast}}\|\tilde f_k\|_{\phi} \\
&\leq \|f-\tilde{f_k}\|_{\phi} + \|\eta_{j_k}\|_{\phi}\|\chi_{\Omega_{k}}\|_{\phi^{\ast}}(\|f\|_{\phi} +1).
\end{align*}
The first term on the right-hand side converges to $0$ when $k \to \infty$. On the other hand, since all domains $\Omega_k$ have finite measure, \azero{} implies that $\chi_{\Omega_k} \in L^{\phi^{\ast}}(\Omega)$ \cite[Theorem 3.7.13]{HarH_book}. Exploiting \eqref{eq:vanishing1}, we can choose the sequence $(j_k)$ in such a way that $\|\eta_{j_k}\|_{\phi} \|\chi_{\Omega_k}\|_{\phi^{\ast}} \leq 2^{-k}$. This choice ensures us that the last term on the previous display converges to $0$ as $k \to \infty$. All together, we have shown that $f_k \to f$ in $L^{\phi}(\Omega)$.\end{proof}
 
Let us now prove the required inequality for the sharp maximal function.
\begin{lem}\label{lem:sharp-norm}
If $\phi \in \Phi_w(\Omega)$ satisfies \azero{}, \aone{}, \atwo{}, \ainc{} and \adec{}, then
\begin{align*}
\|f\|_{L^{\phi}(\Rn)} \lesssim \|M^{\sharp}f\|_{L^{\phi}(\Rn)}
\end{align*}
for all $f \in L^{\phi}(\Rn)$.
\end{lem}

\begin{proof}
By the pointwise inequality $M^{\sharp}f(x) \leq 2 Mf(x)$ and the assumptions on $\phi$, the sharp maximal operator is bounded in $L^{\phi}(\Rn)$. This implies that $M^{\sharp}f_j \to M^{\sharp}f$ in $L^{\phi}(\Rn)$ and $Mg_j \to Mg$ in $L^{\phi^{\ast}}(\Rn)$ for all convergent sequences $(f_j) \subset C^{\infty}_0(\Rn)$ and $(g_j)\subset C^{\infty}_{0,0}(\Rn)$. According to \cite[Lemma 6.2.3]{DieHHMS09}, for these functions we have the estimate
\begin{align*}
|\langle f_j, g_j\rangle| := \left | \int_{\Rn} f_j(x)g_j(x) \, dx \right | \lesssim \int_{\Rn} M^{\sharp}f_j(x) M g_j(x) \,dx = \langle M^{\sharp}f_j, Mg_j\rangle.
\end{align*}
Therefore for general $f \in L^{\phi}(\Rn)$ and $g\in L^{\phi^{\ast}}(\Rn)$ we have
\begin{align*}
|\langle f,g\rangle| &= \lim_{j \to \infty} |\langle f_j,g_j \rangle| \lesssim \lim_{j \to \infty} \langle M^{\sharp} f_j, Mg_j \rangle  = \langle M^{\sharp} f, Mg \rangle \lesssim \|M^{\sharp}f\|_{\phi} \|Mg\|_{\phi^{\ast}} \\
&\lesssim \|M^{\sharp}f\|_{\phi} \|g\|_{\phi^{\ast}}.
\end{align*}
Taking supremum over functions $g$ satisfying $\|g\|_{L^{\phi^{\ast}}(\Rn)} \leq 1$ the claim follows by the Norm conjugate formula, Lemma \ref{lem:norm-conjugate}.
\end{proof}

The next lemma has been proven by Shirai in the case of $[I_\alpha,b]$ but the proof works without changes also for the operator $I_{b, \alpha}$.
\begin{lem}[Lemma 4.2 in \cite{Shi06}]\label{lem:riesz-pointwise}
Suppose $0 < \alpha < n$ and $s \in (1,\infty)$. Then for every $b \in \BMO(\Rn)$ and $f \in C^{\infty}_0(\Rn)$ we have pointwise estimates
\begin{enumerate}
\item $M^{\sharp}([I_\alpha,b]f)(x) \lesssim \|b\|_{\ast} \left ( I_\alpha |f|(x) + I_{\alpha s}(|f|^{s})(x)^{1/s}\right );$
\item $M^{\sharp}(I_{b,\alpha}f)(x) \lesssim \|b\|_{\ast} \left ( I_\alpha |f|(x) + I_{\alpha s}(|f|^{s})(x)^{1/s}\right ).$
\end{enumerate}
\end{lem}

The following Lemma is the main result of \cite{HarH_Riesz} taking account Remark \ref{rem:Riesz-space}.
\begin{lem}\label{lem:riesz-bdd}
If $\phi$ satisfies Assumptions \ref{assu:fractional}, then the Riesz potential $I_\alpha: L^{\phi}(\Rn) \to L^{\psi}(\Rn)$ is bounded.
\end{lem}

The following Theorem shows boundedness of commutators related to Riesz potentials.

\begin{thm}\label{thm:riesz}
Let $1< p \leq q < \infty$, $\frac{\alpha}{n} = \frac{1}{p} - \frac{1}{q}$ and $\phi, \psi \in \Phi_w(\Rn)$ be such that
\begin{equation*}
\phi^{-1}(x,t) \approx t^{\frac{\alpha}{n}}\psi^{-1}(x,t).
\end{equation*}
Let $r \in  (\frac{\alpha}{n}, \frac{1}{p}]$. Assume that $\phi$ satisfies \azero{}, \aone{}, \atwo{}, \ainc{p} and \adec{1/r}. Then we have
\begin{enumerate}
\item $\|[I_\alpha, b]f\|_{L^{\psi}(\Rn)} \lesssim \|b\|_{\ast} \|f\|_{L^{\phi}(\Rn)}$
\item $\|I_{b,\alpha}f\|_{L^{\psi}(\Rn)} \lesssim \|b\|_{\ast} \|f\|_{L^{\phi}(\Rn)}$
\end{enumerate}
for every $f \in L^{\phi}(\Rn)$ and $b \in \BMO(\Rn)$.
\end{thm}

\begin{proof}
We give the proof for $(1)$ as the proof of $(2)$ is identical. Let us denote $\phi_s = \phi\left (x,t^{1/s}\right )$ for any weak $\Phi$-function. Recall that $\psi$ satisfies similar assumptions to $\phi$ but with different parameters in \ainc{} and \adec{}. Therefore, choosing any $s\in (1, p)$ and applying first Lemma \ref{lem:sharp-norm} and then Lemma \ref{lem:riesz-pointwise}, we see that for any $f \in C^{\infty}_0(\Rn)$ and $b \in \BMO(\Rn)$ 
\begin{align}\label{eq:Riesz-1}
\begin{split}
\|[I_\alpha,b]f\|_{\psi} & \lesssim \|M^{\sharp}([I_\alpha,b]f)\|_{\psi} \lesssim \|b\|_{\ast} \|I_\alpha|f| + I_{\alpha s}(|f|^s)^{1/s}\|_{\psi} \\
&\leq  \|b\|_{\ast}\left ( \|I_\alpha|f|\|_{\psi} + \|I_{\alpha s}(|f|^s)\|^{1/s}_{\psi_{s}} \right).
\end{split}
\end{align}
By Lemma \ref{lem:riesz-bdd} $I_\alpha: L^{\phi}(\Rn) \to L^{\psi}(\Rn)$ is bounded. Let us show boundedness of $I_{\alpha s}: L^{\phi_s}(\Rn) \to L^{\psi_s}(\Rn)$. 
It is easy to check that $\psi_s^{-1}(x,t) = \psi^{-1}(x,t)^{s}$. Next, it immediately follows that
\begin{align*}
\psi_{s}^{-1}(x,t) \approx \psi^{-1}(x,t)^{s} \approx t^{-\frac{\alpha s}{n}} \phi^{-1}(x,t)^{s} \approx t^{-\frac{\alpha s}{n}} \phi_{s}^{-1}(x,t).
\end{align*}
In \cite[Proposition 5.2.2]{HarH_book} it is shown that $\phi_s$ satisfies \azero{}, \aone{}, \atwo{}, \ainc{p/s} and \adec{1/rs}. Since $\frac{p}{s}>1$ and $rs \in (\frac{\alpha s}{n}, \frac{s}{p}]$, Lemma \ref{lem:riesz-bdd} shows that $I_{\alpha s}$ is bounded from $L^{\phi_s}(\Rn)$ to $L^{\psi_s}(\Rn)$ and likewise
\begin{align*}
\|I_{\alpha s}(|f|^{s})\|_{\psi_t}^{\frac{1}{s}} \lesssim \||f|^s\|_{\phi_s}^{\frac{1}{s}}.
\end{align*}
Coming back to \eqref{eq:Riesz-1} and definition of $\phi$-norm, we have
\begin{align*}
\|[I_\alpha,b]f\|_{\psi} \lesssim \|b\|_\ast \left (\|f\|_{\phi} +\|f^s\|^{1/s}_{\phi_s}\right ) \lesssim \|b\|_{_\ast} \|f\|_{\phi}.
\end{align*}

Since $\phi$ satisfies \azero{} and \adec{}, we know that $C^{\infty}_0(\Rn)$ is dense in $L^{\phi}(\Rn)$. Now we get the result for all $f \in L^{\phi}(\Rn)$ since the commutator is linear and bounded in a dense set.
\end{proof}

\section{Maximal commutators}
In this section we show boundedness of both maximal commutators $M_b: {L^\phi}(\Rn) \to L^{\phi}(\Rn)$ and $M_{b,\alpha}:L^{\phi}(\Rn) \to L^{\psi}(\Rn)$. We start with the simpler non-fractional version but first we collect some lemmas. The first Lemma was essentially proven in \cite[Corollaries 1.11, 1.12]{AgcGKM15}. Authors' main results lead to pointwise estimates
\begin{align*}
M_b f(x) \lesssim \|b\|_{\ast}M^2f(x) \quad \text{and} \quad |[M,b]f(x)| \lesssim (\|b^{+}\|_{\ast} + \|b^{-}\|_{\infty}) M^2f(x)
\end{align*}
where $b \in \BMO(\Rn)$, so it holds for any (quasi) Banach space where $M$ is bounded, especially for $L^{\phi}(\Rn)$. Similar approach cannot directly be used for $M_{b,\alpha}$ since this operator is not bounded from $L^{\phi}(\Rn)$ to itself.

\begin{lem}\label{lem:maximal-commutator}
If $b \in \BMO(\Rn)$ and $\phi \in \Phi_w(\Rn)$ satisfies \azero{}, \aone{}, \atwo{} and \ainc{}, then 
\begin{align*}
\|M_bf\|_\phi \lesssim \|b\|_{\ast} \|f\|_\phi.
\end{align*}
If additionally $b^{-}\in L^{\infty}(\Rn)$, then also the commutator $[M,b]: L^{\phi}(\Rn) \to L^{\phi}(\Rn)$ is bounded.
\end{lem}

It is straightforward to prove that $b \in \BMO(\Rn)$ is also a necessary condition for boundedness of the maximal commutator and we get the following characterization of $\BMO(\Rn)$.

\begin{prop}
Let $b \in L^{1}_\loc(\Rn)$ and $\phi \in \Phi_w(\Rn)$ satisfy \azero{}, \aone{}, \atwo{} and \ainc{}. Then $M_b: L^{\phi}(\Rn) \to L^{\phi}(\Rn)$ is bounded if and only if $b\in \BMO(\Rn)$.
\end{prop}

\begin{proof}
If $b \in \BMO(\Rn)$, the claim follows directly from Lemma \ref{lem:maximal-commutator}.

Let us then assume that $M_b: L^{\phi}(\Rn) \to L^{\phi}(\Rn)$ is bounded. Take any cube $Q\subset \Rn$. Using H\"older's inequality (Lemma \ref{lem:holder}), the assumption and Lemma \ref{lem:norm-of-ball} in this order, we see that
\begin{align*}
\dfrac{1}{|Q|} \int_{Q}|b(x)-b_Q| \, dx &\leq \dfrac{1}{|Q|} \int_{Q} \left ( \int_{Q} |b(x)-b(y)|\chi_{Q}(y) \, dy \right ) dx \\
&\leq \dfrac{1}{|Q|} \int_{\Rn} M_b\chi_Q(x) \, dx \lesssim \dfrac{1}{|Q|} \|M_b\chi_Q\|_{\phi} \|\chi_Q\|_{\phi^{\ast}} \\
&\leq \dfrac{1}{|Q|} \|\chi_{Q}\|_{\phi} \|\chi_{Q}\|_{\phi^\ast} \leq C.
\end{align*}
Taking the supremum over all cubes $Q$, we see that $b \in \BMO(\Rn)$.
\end{proof}

As the fractional commutators are closely related, we get the boundedness of maximal fractional commutator also in a similar fashion.

\begin{thm}\label{thm:M_alpha_b-bounded}
Let $b \in L^1_\loc(\Rn)$, $1 < p \leq q < \infty$, $\frac{\alpha}{n} = \frac{1}{p} - \frac{1}{q}$ and $\phi, \psi \in \Phi_w(\Rn)$ be such that
\begin{equation*}
\phi^{-1}(x,t) \approx t^{\frac{\alpha}{n}}\psi^{-1}(x,t).
\end{equation*}
Let $r \in  (\frac{\alpha}{n}, \frac{1}{p}]$. Assume that $\phi$ satisfies \azero{}, \aone{}, \atwo{}, \ainc{p} and \adec{1/r}. Then $M_{b,\alpha}: L^{\phi}(\Rn) \to L^{\psi}(\Rn)$ is bounded and
\begin{align*}
\|M_{b,\alpha} f\|_{\psi} \lesssim \|b\|_{\ast}\|f\|_{\phi}
\end{align*}
if and only if $b \in \BMO(\Rn)$.
\end{thm}

\begin{proof}
Let us first assume that $b \in \BMO(\Rn)$. Since we have the pointwise estimate
\begin{align*}
M_{b,\alpha}f \lesssim I_{b,\alpha}(|f|),
\end{align*}
Theorem \ref{thm:riesz} (2) immediately yields $\|M_{b,\alpha} f\|_{\psi} \lesssim \|I_{b,\alpha}(|f|)\|_{\psi} \lesssim  \|b\|_\ast \|f\|_{\phi}$.

Let us then assume that $M_{b, \alpha}$ is bounded and $Q$ is an arbitrary cube. This and H\"older's inequality give
\begin{align*}
\dfrac{1}{|Q|} \int_{Q}|b(x) - b_Q| \, dx &\leq \dfrac{1}{|Q|^{1+\frac{\alpha}{n}}} \int_{Q}\left ( \dfrac{1}{|Q|^{1-\frac{\alpha}{n}}} \int_{Q} |b(x)-b(y)| \chi_Q(y) \, dy\right ) \, dx \\
&\leq \dfrac{1}{|Q|^{1+\frac{\alpha}{n}}} \int_{Q} M_{b,\alpha}\chi_Q(x) \, dx \\
&\leq \dfrac{1}{|Q|^{1+\frac{\alpha}{n}}} \|M_{b,\alpha}\chi_Q\|_{\psi} \|\chi_Q\|_{\psi^{\ast}} \\
&\lesssim \dfrac{1}{|Q|^{1+\frac{\alpha}{n}}} \|\chi_Q\|_{\phi} \|\chi_Q\|_{\psi^\ast}.
\end{align*}
Now we need to estimate $\|\chi_Q\|_{\phi}$ in terms of $\|\chi_Q\|_{\psi}$. This is done with the help of Lemma \ref{lem:measure-scaling}:
\begin{equation}\label{eq:chi-phi-psi}
\|\chi_Q\|_{\phi} \approx \dfrac{1}{\fint_{Q} \phi^{-1}\left (x, \frac{1}{|Q|}\right ) \, dx} \approx \dfrac{1}{\frac{1}{|Q|^{\alpha/n}}\fint_{Q} \psi^{-1}\left (x, \frac{1}{|Q|}\right ) \, dx} \approx |Q|^{\alpha/n}\|\chi_Q\|_{\psi}.
\end{equation}
Therefore Lemma \ref{lem:norm-of-ball} and the previous two displays yield
\begin{align*}
\dfrac{1}{|Q|} \int_{Q} |b(x)-b_Q| \, dx \lesssim \frac{1}{|Q|} \|\chi_Q\|_{\psi} \|\chi_Q\|_{\psi^\ast} < C.
\end{align*}
Taking supremum over all cubes $Q$ we see that $b \in \BMO(\Rn)$.
\end{proof}

\section{Commutators of fractional maximal operators}

In this section we combine previous estimates of maximal operators to show boundedness of commutators of fractional maximal operators generalizing the ideas of Zhang, Si and Wu in \cite{varexp1,varexp2}. 

In \cite{BasMR01} the authors showed the following characterization for functions with bounded mean oscillation and negative part. The proof of $b \in \BMO(\Rn)$ is similar to the proof of Lemma \ref{lem:norm-thing}. However, the $L^{1}$ nature of assumption with non-fractional maximal function is needed for $b^{-}\in L^{\infty}(\Rn)$ and it is not known if assumption like the norm estimate in Lemma \ref{lem:norm-thing} with fractional maximal function implies this.
\begin{lem}\label{lem:L1-version}
Suppose $b \in L^1_\loc(\Rn)$ and 
\begin{align*}
\sup_{Q} \dfrac{1}{|Q|} \int_Q |b(x)-M_Qb(x)| \, dx < \infty.
\end{align*}
Then $b \in \BMO(\Rn)$ and $b^{-}\in L^{\infty}(\Rn)$.
\end{lem}

The next result is a weaker form of Proposition \ref{prop:main-2}, but it will be used later when we have commutator with $|b|$ which is clearly non-negative.
\begin{lem}\label{lem:alpha-commutator-bounded}
Let Assumptions \ref{assu:fractional} hold. If $0 \leq b \in \BMO(\Rn)$, then the commutator $[M_{\alpha}, b]: L^{\phi}(\Rn) \to L^{\psi}(\Rn)$ is bounded.
\end{lem}

\begin{proof}
As $|b(x)|= b(x)$ for all $x\in \Rn$, we see that
\begin{align*}
|[M_\alpha,b]f(x)| &= |M_\alpha(bf)(x) - b(x)M_\alpha f(x)| \\
&\leq \left | \sup_{Q \ni x} \dfrac{1}{|Q|^{1-\alpha/n}} \int_{Q} b(y)|f(y)| \, dy - \sup_{Q \ni x} \dfrac{1}{|Q|^{1-\alpha/n}} \int_{Q} b(x)|f(y)| \, dy \right | \\
&\leq \sup_{Q\ni x} \dfrac{1}{|Q|^{1-\alpha/n}} \int_{Q}|b(x)-b(y)||f(y)| \, dy = M_{b,\alpha}f(x).
\end{align*}
By Theorem \ref{thm:M_alpha_b-bounded} we see that the commutator is bounded.
\end{proof}

Let us separate the following result as a Lemma.

\begin{lem}\label{lem:norm-thing}
Let Assumptions \ref{assu:fractional} hold and suppose $[M_\alpha, b]:L^{\phi}(\Rn) \to L^{\psi}(\Rn)$ is bounded, where $b \in L^1_\loc(\Rn)$. Then $b \in \BMO(\Rn)$ and
\begin{align*}
\dfrac{\|(b - |Q|^{-\alpha/n} M_{\alpha,Q}b)\chi_Q\|_{\psi}}{\|\chi_Q\|_{\psi}} < C
\end{align*}
for every cube $Q$ and some finite constant $C$ independent of $Q$.
\end{lem}

\begin{proof}
Let us first prove the norm estimate. First of all, we have the following identities for maximal functions of a characteristic function
\begin{align}\label{eq:M-b-chi1}
M_\alpha \chi_Q(x) = M_{\alpha,Q}\chi_Q(x) = \sup_{Q \ni x} \dfrac{1}{|Q|^{1-\alpha/n}} \int_{Q} \chi_{Q}(x) \, dx = |Q|^{\alpha/n} 
\end{align}
and
\begin{align}\label{eq:M-b-chi2}
M_\alpha(b \chi_Q)(x)  = \sup_{Q \ni x} \dfrac{1}{|Q|^{1-\alpha/n}}\int_{Q}b(x) \chi_Q(x)\, dx = M_{\alpha, Q}b(x).
\end{align}
These combined with the assumption that $[M_\alpha,b]:L^{\phi}(\Rn) \to L^{\psi}(\Rn)$ is bounded, we have the following estimate
\begin{align*}
\|(b-|Q|^{-\alpha/n}M_{\alpha,Q}(b))\chi_Q\|_{\psi} &= |Q|^{-\alpha/n} \|(b|Q|^{\alpha/n} - M_{\alpha,Q}b)\chi_Q\|_{\psi}\\
&\leq |Q|^{-\alpha/n}\|bM_\alpha \chi_Q - M_{\alpha}(b\chi_Q)\|_{\psi} \\
&= |Q|^{-\alpha/n} \|[M_\alpha, b]\chi_Q\|_{\psi} \lesssim |Q|^{-\alpha/n} \|\chi_Q\|_{\phi}.
\end{align*}
We proceed as in \eqref{eq:chi-phi-psi} to estimate $\|\chi_Q\|_{\phi} \lesssim |Q|^{\alpha/n}\|\chi_Q\|_{\psi}$.
Thus
\begin{align*}
\|(b-|Q|^{-\alpha/n}M_{\alpha,Q}b)\chi_Q\|_{\psi}\lesssim \|\chi_Q\|_{\psi}
\end{align*}
for every cube $Q$, so we have proved the desired norm estimate.

Now we are ready to prove that $b \in \BMO(\Rn)$. Let us denote for any cube $Q$ sets $E:=\{x \in Q : b(x) \leq b_Q\}$ and $F:= \{x \in Q : b(x) > b_Q\}$. It is immediate that $E = F^c$ and this combined with
\begin{align*}
\int_{Q}b(x) - b_Q \, dx=0
\end{align*}
yields
\begin{align*}
\int_{E} |b(x) - b_Q|\, dx = \int_{F}|b(x) - b_Q|\, dx.
\end{align*}
Also, for any $x \in E$ it holds that $b(x) \leq b_Q \leq |b_Q| \leq |Q|^{-\alpha/n}M_{\alpha,Q} b(x)$. Therefore
\begin{align}
|b(x)-b_Q| \leq |b(x) - |Q|^{-\alpha/n}M_{\alpha,Q}(b)(x)|
\end{align}
in the set $E$.

Let us estimate the mean oscillation of the function $b$
\begin{align*}
\dfrac{1}{|Q|} \int_Q |b(x)-b_Q|\, dx &= \dfrac{2}{|Q|}\int_{E}|b(x)-b_Q| \, dx \leq \dfrac{2}{|Q|} \int_{E} |b(x) - |Q|^{-\alpha/n}M_{\alpha,Q}b(x)|\, dx \\
&\leq \dfrac{2}{|Q|} \int_Q |b(x) - |Q|^{-\alpha/n}M_{\alpha,Q}b(x)| \, dx.
\end{align*}
We estimate the last integral with H\"older's inequality, use the previously proven norm estimate and Lemma \ref{lem:norm-of-ball} to get
\begin{align}\label{eq:bmo-ineq-line}
\begin{split}
\dfrac{1}{|Q|} \int_Q |b(x)-b_Q|\, dx &\lesssim \dfrac{1}{|Q|} \|(b-|Q|^{-\alpha/n}M_{\alpha,Q}b)\chi_Q\|_{\psi} \|\chi_Q\|_{\psi^{\ast}} \\
&\lesssim \dfrac{1}{|Q|} \|\chi_Q\|_{\psi} \|\chi_Q\|_{\psi^{\ast}} \leq C.
\end{split}
\end{align}
Thus, by definition, $b \in \BMO(\Rn)$.
\end{proof}

In the next Proposition we characterize $\BMO(\Rn)$ in terms of boundedness of commutators with fractional maximal operators. In Theorem \ref{thm:main-2} we improve condition $(3)$ to a more intrinsic condition containing the fractional maximal operator. The boundedness of the standard maximal commutator $[M, b]: L^{\phi}(\Rn) \to L^{\phi}(\Rn)$ if and only if $b\in \BMO(\Rn)$ and $b^{-}\in L^{\infty}(\Rn)$ was covered in Lemma \ref{lem:maximal-commutator}.
\begin{prop}\label{prop:main-2}
Let Assumptions \ref{assu:fractional} hold. If $b \in L^{1}_{\loc}(\Rn)$, then the following are equivalent:
\begin{enumerate}
\item{$b \in \BMO(\Rn)$ and $b^{-} \in L^{\infty}(\Rn)$.}
\item{The commutator $[M_{\alpha}, b]: L^{\phi}(\Rn) \to L^{\psi}(\Rn)$ is bounded.}
\item{We have
\begin{equation*}
\sup_{Q} \dfrac{\|(b-M_Q b)\chi_Q\|_{L^{\psi}(\Rn)}}{\|\chi_Q\|_{L^{\psi}(\Rn)}} < \infty.
\end{equation*}}
\end{enumerate}
\end{prop}

\begin{proof}
Recall that the assumptions on $\phi$ guarantee the existence of $\psi$ which satisfies \azero{}--\atwo{}, \ainc{q}, \adec{1/(r-\alpha/n)} and the relation between their inverses \cite[Lemma 5.2.3]{HarH_book}. We proceed to prove $(1) \Rightarrow (2) \Rightarrow (3) \Rightarrow (1)$. For the first implication we start with the following pointwise estimates
\begin{align*}
|[M_\alpha, b]f(x) &- [M_\alpha, |b|]f(x)| \\
&= |M_\alpha(bf)(x) - b(x)M_\alpha f(x) - M_\alpha(|b|f)(x) + |b(x)|M_\alpha f(x)|  \\
&\leq |M_\alpha (bf)(x) - M_\alpha(|b|f)(x)| + |2b^{-}(x)M_\alpha f(x)| \\
&= 2b^{-}(x) M_\alpha f(x),
\end{align*}
where the inequality follows directly from $|b(x)|-b(x) = 2b^{-}(x)$. In other words, we have a direct estimate for the commutator
\begin{align*}
|[M_\alpha,b]f(x)| &\leq |[M_\alpha, b]f(x) - [M_\alpha, |b|]f(x)| + |[M_\alpha,|b|]f(x)| \\
&\leq 2b^{-}(x) M_\alpha f(x) + |[M_\alpha, |b|]f(x)|.
\end{align*}
Since also $|b| \in \BMO(\Rn)$, recalling $b^{-} \in L^{\infty}(\Rn)$ we can combine Theorem \ref{thm:M_alpha-bounded} and Lemma \ref{lem:alpha-commutator-bounded} to get
\begin{align*}
\|[M_\alpha,b]f\|_{\psi} \leq 2 \|b^{-}\|_{\infty} \|M_\alpha f\|_{\psi} + \|[M_\alpha,|b|]f\|_{\psi} \lesssim \|f\|_{\phi}.
\end{align*}

Next we prove $(2) \Rightarrow (3)$ and start with Lemma \ref{lem:norm-thing}, which yields
\begin{align}\label{eq:main-eq-1}
\begin{split}
\|(b- M_Qb)\chi_Q\|_{\psi} &= \|(b - |Q|^{-\frac{\alpha}{n}}M_{\alpha,Q}b + |Q|^{-\frac{\alpha}{n}}M_{\alpha,Q}b  - M_Qb)\chi_Q\|_{\psi} \\
&\leq \|(b - |Q|^{-\frac{\alpha}{n}}M_{\alpha,Q}b)\chi_Q\|_{\psi} + \| (|Q|^{-\frac{\alpha}{n}}M_{\alpha,Q}b - M_Qb)\chi_Q\|_{\psi} \\
&\lesssim \|\chi_Q\|_{\psi} + |Q|^{-\frac{\alpha}{n}} \|(M_{\alpha,Q}b - |Q|^{\frac{\alpha}{n}} M_Qb)\chi_Q\|_{\psi}.
\end{split}
\end{align}
Letting $x \in Q$ and taking account \eqref{eq:M-b-chi1} and \eqref{eq:M-b-chi2} we have
\begin{align*}
[M_\alpha, |b|]\chi_Q(x) &= M_\alpha(b\chi_Q)(x) - |b(x)| M_\alpha \chi_Q(x) = M_{\alpha,Q}b(x) - |b(x)|M_\alpha \chi_Q(x) \\
&=M_{\alpha,Q}b(x) - |Q|^{\frac{\alpha}{n}}|b(x)|.
\end{align*}
together with
\begin{align*}
|Q|^{\frac{\alpha}{n}}[M,|b|]\chi_Q(x) &= |Q|^{\frac{\alpha}{n}} M_Qb(x) - |Q|^{\frac{\alpha}{n}} |b(x)| M\chi_Q(x) \\
&= |Q|^{\frac{\alpha}{n}} M_Qb(x) - |Q|^{\frac{\alpha}{n}} |b(x)|.
\end{align*}
We subtract the last display from the previous and get
\begin{align*}
[M_\alpha, |b|]\chi_Q(x) &- |Q|^{\frac{\alpha}{n}}[M,|b|]\chi_Q(x) =M_{\alpha,Q}b(x) - |Q|^{\frac{\alpha}{n}} M_Qb(x).
\end{align*}
This and \eqref{eq:main-eq-1} combine to
\begin{align*}
\|(b- M_Qb)\chi_Q\|_{\psi}  &\lesssim \|\chi_{Q}\|_{\psi} + |Q|^{-\frac{\alpha}{n}} \| [M_\alpha, |b|]\chi_Q- |Q|^{\frac{\alpha}{n}}[M,|b|]\chi_Q\|_{\psi} \\
&\leq \|\chi_Q\|_{\psi}+ |Q|^{-\frac{\alpha}{n}} \|[M_\alpha, |b|]\chi_Q\|_{\psi} +  \|[M,|b|]\chi_Q\|_{\psi}.
\end{align*}
As $0\leq |b| \in \BMO(\Rn)$, the commutators are bounded (Lemma \ref{lem:maximal-commutator} and Lemma \ref{lem:alpha-commutator-bounded}), and thus
\begin{align*}
\|(b- M_Qb)\chi_Q\|_{\psi}  \lesssim \|\chi_Q\|_{\psi} + |Q|^{-\frac{\alpha}{n}}\|\chi_Q\|_{\phi} + \|\chi_Q\|_{\psi}.
\end{align*}
Finally, using \eqref{eq:chi-phi-psi} for the second term on the right-hand side, we have
\begin{align*}
\|(b- M_Qb)\chi_Q\|_{\psi}  \lesssim \|\chi_Q\|_{\psi}.
\end{align*}

Then we show $(3) \Rightarrow (1)$.  H\"older's inequality, the assumption and Lemma \ref{lem:norm-of-ball} yield

\begin{align*}
\dfrac{1}{|Q|} \int_{Q} |b(x) - M_Qb(x)| \, dx &\lesssim \dfrac{1}{|Q|} \|(b- M_Qb)\chi_Q\|_{\psi} \|\chi_Q\|_{\psi^{\ast}} \\
&\lesssim \dfrac{1}{|Q|} \|\chi_Q\|_{\psi} \|\chi_Q\|_{\psi^{\ast}} \leq C.
\end{align*}
Now Lemma \ref{lem:L1-version} shows that $b \in \BMO(\Rn)$ and $b^{-} \in L^{\infty}(\Rn)$.
\end{proof}

Now we prove our main theorem regarding commutators of fractional maximal operators. The main difference to previous Proposition is that the norm condition includes fractional maximal operator instead of a standard Hardy--Littlewood maximal operator. Also, it is enough for the norm condition to hold for any weak $\Phi$-function $\eta$ with structural conditions guaranteeing boundedness of the Hardy--Littlewood maximal operator. 
\begin{thm}\label{thm:main-2}
Let $1 < p \leq q < \infty$, $\frac{\alpha}{n} = \frac{1}{p} - \frac{1}{q}$ and $\phi, \psi \in \Phi_w(\Rn)$ be such that
\begin{equation*}
\phi^{-1}(x,t) \approx  t^{\frac{\alpha}{n}}\psi^{-1}(x,t).
\end{equation*}
Let $r \in  (\frac{\alpha}{n}, \frac{1}{p}]$. Assume that $\phi$ satisfies \azero{}, \aone{}, \atwo{}, \ainc{p} and \adec{1/r}. If $b \in L^{1}_{\loc}(\Rn)$, then the following are equivalent:
\begin{enumerate}
\item{$b \in \BMO(\Rn)$ and $b^{-} \in L^{\infty}(\Rn)$.}
\item{The commutator $[M_{\alpha}, b]: L^{\phi}(\Rn) \to L^{\psi}(\Rn)$ is bounded.}
\item{We have
\begin{equation}\label{eq:main-eq}
\sup_{Q} \dfrac{\|(b-|Q|^{-\alpha/n}M_{\alpha,Q} b)\chi_Q\|_{L^{\eta}(\Rn)}}{\|\chi_Q\|_{L^{\eta}(\Rn)}} < \infty
\end{equation}
for some $\eta \in \Phi_w(\Rn)$ for which $M:L^{\eta}(\Rn) \to L^{\eta}(\Rn)$ is bounded.}
\item{Condition \eqref{eq:main-eq} holds for all $\eta \in \Phi_w(\Rn)$ satisfying \azero{}, \aone{}, \atwo{}, \ainc{} and \adec{}.}
\end{enumerate}
\end{thm}

\begin{proof}
Equivalence $(1) \Leftrightarrow (2)$ was proven in Proposition \ref{prop:main-2}. As the implication $(4) \Rightarrow (3)$ is immediate, we proceed to prove $(3) \Rightarrow (1)$ and $(2) \Rightarrow (4)$.

We first assume $(3)$. Choosing any cube $Q$ we get
\begin{align*}
\dfrac{1}{|Q|}\int_{Q} |b(x) - M_Qb(x)| \, dx &\leq \dfrac{1}{|Q|} \int_{Q} |b(x) - |Q|^{-\alpha/n}M_{\alpha,Q}b(x)| \, dx \\
&\quad + \dfrac{1}{|Q|}\int_{Q} ||Q|^{-\alpha/n}M_{\alpha,Q}b(x) - M_Qb(x)| \, dx =: J_1 + J_2.
\end{align*}
We first estimate the term $J_1$ with H\"older's inequality, Lemma \ref{lem:norm-of-ball} and the assumption as follows
\begin{align*}
J_1 \lesssim \dfrac{1}{|Q|} \|(b-|Q|^{-\alpha/n}M_{\alpha,Q} b)\chi_Q\|_{\psi}\|\chi_Q\|_{\psi^{\ast}} \lesssim \dfrac{\|(b-|Q|^{-\alpha/n}M_{\alpha,Q} b)\chi_Q\|_{\psi}}{\|\chi_Q\|_{\psi}} \leq C.
\end{align*}

For $J_2$ let us choose any $x \in Q$. Then \eqref{eq:M-b-chi1} and \eqref{eq:M-b-chi2} guarantee us that $M\chi_Q(x) = 1$ and $M(b\chi_Q)(x) =M_Qb(x)$ for the maximal function and $M_\alpha \chi_Q(x) = |Q|^{\alpha/n}$ and $M_{\alpha}(b\chi_Q)(x) = M_{\alpha,Q}b(x)$ for the fractional maximal function. With these observations we estimate the integrand in $J_2$ as a sum of commutators
\begin{align*}
\big | |Q|^{-\alpha/n} &M_{\alpha,Q}b(x) - M_Qb(x)\big | \leq \left |Q|^{-\alpha/n}|M_{\alpha,Q}b(x) - |Q|^{\alpha/n}|b(x)|\right | + \left ||b(x) - M_Qb(x)\right | \\
&=|Q|^{-\alpha/n} |M_\alpha (b\chi_Q)(x) - |b(x)|M_\alpha \chi_Q(x)| + ||b(x)|M\chi_Q(x) - M(b\chi_Q)(x)| \\
&=|Q|^{-\alpha/n} |[M_\alpha, |b|]\chi_Q(x)| + |[M,|b|]\chi_Q(x)|.
\end{align*}
This yields
\begin{align*}
J_2 \leq \dfrac{1}{|Q|^{1+\alpha/n}} \int_{Q}|[M_\alpha,|b|]\chi_Q(x)|\, dx + \dfrac{1}{|Q|} \int_{Q} |[M,|b|]\chi_Q(x)| \, dx.
\end{align*}

Now the assumption $(3)$ and a similar argument as in latter part of Lemma \ref{lem:norm-thing} show that $b \in \BMO(\Rn)$ which further implies that $|b| \in \BMO(\Rn)$. Therefore, we can apply Lemma \ref{lem:alpha-commutator-bounded} to show that $[M_\alpha, |b|]$ and $[M,|b|]$ are bounded operators. The first term can be shown to be bounded in the following way using H\"older's inequality, the boundedness of the commutator, \eqref{eq:chi-phi-psi} and Lemma \ref{lem:norm-of-ball} 
\begin{align*}
\dfrac{1}{|Q|^{1+\alpha/n}} \int_{Q}|[M_\alpha,|b|]\chi_Q(x)|\, dx &\lesssim \dfrac{1}{|Q|^{1+\alpha/n}} \|[M_\alpha,|b|]\chi_Q\|_{\psi}\|\chi_Q\|_{\psi^{\ast}} \\
&\lesssim \dfrac{1}{|Q|^{1+\alpha/n}} \|\chi_Q\|_{\phi}\|\chi_Q\|_{\psi^\ast} \\
&\lesssim \dfrac{1}{|Q|} \|\chi_Q\|_{\psi} \|\chi_Q\|_{\psi^{\ast}} \leq C.
\end{align*}
In a similar fashion the second term is shown be bounded
\begin{align*}
\dfrac{1}{|Q|} \int_{Q}|[M,|b|]\chi_Q(x)| \, dx \lesssim \dfrac{1}{|Q|} \|[M,|b|]\chi_Q\|_{\phi} \|\chi_Q\|_{\phi^{\ast}} \lesssim \dfrac{1}{|Q|} \|\chi_{Q}\|_{\phi}\|\chi\|_{\phi^\ast} \leq C.
\end{align*}

Thus we have shown that
\begin{align*}
\dfrac{1}{|Q|} \int_{Q}|b(x) - M_Qb(x)|\, dx < C,
\end{align*}
from which we know due to Lemma \ref{lem:L1-version} that $b \in \BMO(\Rn)$ and $b^{-} \in L^{\infty}(\Rn)$.

Let us then assume $(2)$ and prove $(4)$. Lemma \ref{lem:norm-thing} shows that \eqref{eq:main-eq} holds when $\eta = \psi$ satisfying the relation with $\phi$. Let $\eta$ then be any weak $\Phi$-function such that the maximal operator is bounded from $L^\eta(\Rn) \to L^{\eta}(\Rn)$. 
We choose $\psi(x,t) = \eta(x,t^{r})$, where $r>\frac{n}{n-\alpha}$ to guarantee that $\phi$ satisfies \ainc{}. Let us denote $\xi(x,t) = \eta(x,t^{r'})$, where $r'$ is the H\"older conjugate of $r$. Now
\begin{align*}
\eta^{-1}(x,t) = \eta^{-1}(x,t)^{1/r} \eta^{-1}(x,t)^{1/r'} = \psi^{-1}(x,t)\xi^{-1}(x,t)
\end{align*}
so the generalized H\"older's inequality, Lemma \ref{lem:gen-holder}, is valid with $\eta, \psi$ and $\xi$. Applying first it and then \eqref{eq:main-eq} with $\eta = \psi$ we get
\begin{align*}
\dfrac{\|(b-|Q|^{-\alpha/n}M_{\alpha,Q}b)\chi_Q\|_{\eta}}{\|\chi_Q\|_{\eta}} \lesssim \dfrac{\|(b-|Q|^{-\alpha/n}M_{\alpha,Q}b)\chi_Q\|_{\psi}\|\chi_Q\|_{\xi}}{\|\chi_Q\|_{\eta}} \lesssim \dfrac{\|\chi_Q\|_{\psi}\|\chi_Q\|_{\xi}}{\|\chi_Q\|_{\eta}}.
\end{align*}
From the definition of $\psi$ and $\xi$ we continue
\begin{align*}
\dfrac{\|(b-|Q|^{-\alpha/n}M_{\alpha,Q}b)\chi_Q\|_{\eta}}{\|\chi_Q\|_{\eta}} \lesssim \dfrac{\|\chi_Q\|_{\eta}^{1/r}\|\chi_Q\|_{\eta}^{1/r'}}{\|\chi_Q\|_{\eta}}=1
\end{align*}
so \eqref{eq:main-eq} holds.
\end{proof}


\bigskip

\noindent\small{\textsc{A. Karppinen}}\\
\small{Department of Mathematics and Statistics,
FI-20014 University of Turku, Finland}\\
\footnotesize{\texttt{arttu.a.karppinen@utu.fi}}\\


\begin{thebibliography}{99}

\bibitem{AgcGKM15}
M.\ Agcayazi, A.\ Gogatishvili, K.\ Koca and R.\ Mustafayev:
A note on maximal commutators and commutators of maximal functions,
\emph{J. Math. Soc.} Japan 67 (2015), no. 2, 581--593.

\bibitem{Alm20}
Almeida, Alexandre:
Maximal commutators and commutators of potential operators in new vanishing Morrey spaces,
Nonlinear Anal. 192 (2020), 111684, 12 pp.





\bibitem{BasMR01}
J.\ Bastero, M.\ Milman and F.\ Ruiz:
Commutators for the maximal and sharp functions,
\emph{Proc. Amer. Math. Soc.} 128 (2000), no. 11, 3329--3334.


\bibitem{product}
A.\ Bonami, T.\ Iwaniec, P.\ Jones, and M.\ Zinsmeister:
On the product of functions in $\BMO$ and $H^1$,
Ann. Inst. Fourier (Grenoble) 57 (2007), no. 5, 1405--1439.

\bibitem{fractional3}
C.\ Capone, D.\ Cruz-Uribe and A.\ Fiorenza:
The fractional maximal operator and fractional integrals on variable Lp spaces,
Rev. Mat. Iberoam. 23 (2007), no. 3, 743--770.


\bibitem{ColM15a}
M.\ Colombo and G.\ Mingione: 
Regularity for double phase variational problems,
\emph{Arch. Ration. Mech. Anal.} 215 (2015), no. 2, 443--496.


\bibitem{Cia99}
A.\ Cianchi, Andrea:
Strong and weak type inequalities for some classical operators in Orlicz spaces,
J. London Math. Soc. (2) 60 (1999), no. 1, 187--202.

\bibitem{CoiRW76}
R.\ Coifman, R.\ Rochberg and G.\ Weiss:
Factorization theorems for Hardy spaces in several variables.
Ann. of Math. (2) 103 (1976), no. 3, 611--635.

\bibitem{CruDF08}
D.\ Cruz-Uribe, L.\ Diening and A.\ Fiorenza:
A new proof of the boundedness of maximal operators on variable Lebesgue spaces,
\emph{Boll. Unione Mat. Ital.} (9) 2 (2009), no. 1, 151–173.

\bibitem{extrapolation}
D. Cruz-Uribe and P. Hästö: Extrapolation and interpolation in generalized Orlicz spaces, Trans. Amer. Math. Soc. 370 (2018), no. 6, 4323--4349. 


\bibitem{Die05}
L. Diening: Maximal function on Musielak--Orlicz spaces and generalized Lebesgue spaces,
\emph{Bull. Sci. Math.} 129 (2005), no. 8, 657--700.
%
\bibitem{DieHHMS09}
L. Diening, P. Harjulehto, P. Hästö, Y. Mizuta, T. Shimomura: Maximal functions in variable exponent spaces: limiting cases of the exponent, 
\emph{Ann. Acad. Sci. Fenn. Math.} 34 (2009), no. 2 503--522.




\bibitem{orlicz1}
V.\ Guliyev, F.\ Deringoz and S.\ Hasanov:
Fractional maximal function and its commutators on Orlicz spaces,
Anal. Math. Phys. 9 (2019), no. 1, 165--179.

\bibitem{morrey2}
V.\ Guliyev, S.\  Aliyev, T.\  Karaman and P.\ Shukurov:
Boundedness of sublinear operators and commutators on generalized Morrey spaces,
Integral Equations Operator Theory 71 (2011), no. 3, 327--355.


\bibitem{HarH_Riesz}
P. Harjulehto and P. Hästö: The Riesz potential in generalized Orlicz spaces, \emph{Forum Math.} 29 (2017), no. 1, 229--244.
%
\bibitem{HarH_book}
P.\ Harjulehto and P.\ H\"ast\"o:
\emph{Orlicz spaces and Generalized Orlicz spaces},
Lecture Notes in Mathematics, vol. 2236, Springer, Cham, 2019.
%
%
\bibitem{Has15}
P. Hästö: The maximal operator on generalized Orlicz spaces, 
\emph{J. Funct. Anal.} 269 (2015), no. 12, 4038--4048.
%
%
%
\bibitem{Hed72}
L.\ Hedberg:
On certain convolution inequalities, \emph{Proc. Amer. Math. Soc.} 36 (1972), 505--510.

%

\bibitem{p(x)-riesz}
M.\ Izuki:
Commutators of fractional integrals on Lebesgue and Herz spaces with variable exponent,
Rend. Circ. Mat. Palermo (2) 59 (2010), no. 3, 461--472.

%
\bibitem{Jan78}
S.\ Janson:
Mean oscillation and commutators of singular integral operators,
Ark. Mat. 16 (1978), no. 2, 263--270.


\bibitem{fractional2}
V.\ Kokilashvili and S.\ Samko:
On Sobolev theorem for Riesz-type potentials in Lebesgue spaces with variable exponent,
Z. Anal. Anwendungen 22 (2003), no. 4, 899--910.



\bibitem{fractional1}
B.\ Muckenhoupt and R.\ Wheeden:
Weighted norm inequalities for fractional integrals,
Trans. Amer. Math. Soc. 192 (1974), 261--274.


\bibitem{campanato}
S.\ Shi and S.\ Lu:
Some characterizations of Campanato spaces via commutators on Morrey spaces,
Pacific J. Math. 264 (2013), no. 1, 221--234.

\bibitem{Shi06}
S.\ Shirai, Necessary and sufficient conditions for boundedness of commutators of fractional integral operators on classical Morrey spaces, \emph{Hokkaido Math. J.} 35 (2006), 683--696.


\bibitem{morrey1}
C.\ Xie, Chunping:
Some estimates of commutators,
Real Anal. Exchange 36 (2010/11), no. 2, 405--415.

\bibitem{Lp-fractional}
P.\ Zhang, Z.\ Si and J.\ Wu:
Commutators of fractional maximal functions,
Acta Math. Sinica (Chin. Ser.) 52 (2009), no. 6, 1235--1238.

\bibitem{varexp1}
P.\ Zhang, Z.\ Si and J.\ Wu:
Some notes on commutators of the fractional maximal function on variable Lebesgue spaces,
J. Inequal. Appl. 2019, Paper No. 9, 17 pp.

\bibitem{varexp2}
P.\ Zhang and J.\ Wu, Jianglong:
Commutators of the fractional maximal function on variable exponent Lebesgue spaces,
Czechoslovak Math. J. 64(139) (2014), no. 1, 183--197.

\end{thebibliography}
\end{document}